\documentclass{amsart}
\usepackage[english]{babel}
\usepackage[latin1]{inputenc}
\usepackage[dvips,final]{graphics}
\usepackage{amsmath,amsfonts,amssymb,amsthm,amscd,array,stmaryrd,mathrsfs}
\usepackage{pstricks}
 \usepackage[all]{xy}
\usepackage{textcomp}
 \usepackage[final]{epsfig}

\theoremstyle{plain}
\newtheorem{thm}{Theorem}
\newtheorem{lem}{Lemma}[section]
\newtheorem{cor}[lem]{Corollary}
\newtheorem{prop}[thm]{Proposition}

\theoremstyle{definition}
\newtheorem{defi}[lem]{Definition}
\newtheorem{rem}[lem]{Remark}
\newtheorem{ex}[lem]{Example}

\newcommand{\R}{\mathbb{R}}
\newcommand{\Z}{\mathbb{Z}}
\newcommand{\C}{\mathbb{C}}
\newcommand{\CP}{\mathbb{CP}}
\newcommand{\bbH}{\mathbb{H}}

\newcommand{\K}{\mathbb{K}}

\newcommand{\A}{\mathcal{A}}

\newcommand{\G}{\mathcal{G}}

\newcommand{\asl}{\mathrm{asl}}
\newcommand{\Cl}{\mathrm{Cl}}
\newcommand{\M}{\mathrm{M}}
\newcommand{\osp}{\mathrm{osp}}

\newcommand{\half}{\frac{1}{2}}
\newcommand{\End}{\textup{End}}

\newcommand{\e}{\varepsilon}

\def\a{\alpha}
\def\b{\beta}

\def\e{\varepsilon}

\def\g{\gamma}
\def\G{\Gamma}

\def\om{\omega}

\def\l{\lambda}
\def\m{\mu}

\begin{document}

\title{Simple graded commutative algebras}

\author{Sophie Morier-Genoud
\hskip 1cm
Valentin Ovsienko}

\address{
Sophie Morier-Genoud,
Universit\'e Paris Diderot Paris 7,
UFR de math\'ematiques case 7012,
75205 Paris Cedex 13, France}

\address{
Valentin Ovsienko,
CNRS,
Institut Camille Jordan,
Universit\'e Claude Bernard Lyon~1,
43 boulevard du 11 novembre 1918,
69622 Villeurbanne cedex,
France}

\email{sophiemg@math.jussieu.fr,
ovsienko@math.univ-lyon1.fr}

\date{}


\begin{abstract}
We study the notion of $\G$-graded commutative algebra
for an arbitrary abelian group $\G$.
The main examples are the Clifford algebras
already treated in \cite{AM1}.
We prove that the Clifford algebras are the only
simple finite-dimensional associative
graded commutative algebras over $\R$ or $\C$.
Our approach also leads to
non-associative graded commutative algebras
extending the Clifford algebras.
\end{abstract}

\maketitle

\thispagestyle{empty}

\section{Introduction and the Main Theorems}

Let $\left(\G,+\right)$ be an abelian group and 
$\langle\,,\,\rangle:\G\times\G\to\Z_2$ a bilinear map.
An algebra $\A$ is called a $\G$-graded commutative
(or $\G$-commutative for short) if $\A$ is $\G$-graded in the usual sense:
$$
\A=\bigoplus_{\g\in\G}\A_\g
$$
such that $\A_\g\cdot\A_{\g'}\subset\A_{\g+\g'}$ and for all 
homogeneous elements $a,b\in\A$
one has
\begin{equation}
\label{ComPr}
a\,{}b=(-1)^{\left\langle\bar{a},\bar{b}\right\rangle}\,b\,{}a.
\end{equation}
where $\bar{a}$ and $\bar{b}$ are the corresponding degrees.

If $\G=\Z_2$ and $\langle\,,\,\rangle$ is the standard product,
then the above definition coincides with that of ``supercommutative algebra''.
The main examples of associative supercommutative algebras
are: the algebras of differential forms on manifolds or, more generally,
algebras of functions on supermanifolds.
These algebras cannot be simple
(i.e., they always contain a non-trivial proper ideal).

Study of $\G$-commutative algebras is a quite new subject.
We cite here the pioneering work of Albuquerque and Majid,
\cite{AM}, where the classical algebra of quaternions $\bbH$
is understood as a $\Z_2\times\Z_2$-commutative algebra.
This result was generalized for the Clifford algebras in \cite{AM1}.
In \cite{BMZ} a more general notion of $\b$-commutative algebra 
is considered and the structure of simple algebras is completely
determined.
An application of graded commutative algebras to 
deformation quantization is recently proposed in \cite{Phy}.

Ih this paper we also consider the quaternion algebra and the
Clifford algebras as graded commutative algebras.
Our grading is slightly different from that of \cite{AM,AM1}.
This difference concerns essentially the parity of the elements.

\subsection*{The first example.}
The starting point of this work is the following
observation, see \cite{MO}.
The quaternion algebra $\bbH$ is
$\Z_2\times\Z_2\times\Z_2$-commutative in the following sense.
Associate the ``triple degree'' to the standard basis elements of $\bbH$:
\begin{equation}
\label{DegH}
\begin{array}{rcl}
\bar{\e}&=&(0,0,0),\\[4pt]
\bar{i}&=&(0,1,1),\\[4pt]
\bar{j}&=&(1,0,1),\\[4pt]
\bar{k}&=&(1,1,0),
\end{array}
\end{equation}
where $\e$ denotes the unit.
The usual product of quaternions then
satisfies the condition (\ref{ComPr}),
where $\langle\,,\,\rangle$ is the usual \textit{scalar product} of 3-vectors.
Indeed, $\left\langle\bar{i},\,\bar{j}\right\rangle=1$ and similarly for $k$,
so that $i,j$ and $k$ anticommute with each other.
But,
$\left\langle\bar{i},\,\bar{i}\right\rangle=0$,
so that $i,j,k$ commute with themselves.

In fact, the first two components in (\ref{DegH}) contain the full information of the grading
of $\bbH$.
By ``forgetting'' the third component,
one obtains a $\Z_2\times\Z_2$-grading on the quaternions:
$$
\bar{1}=(0,0),\qquad
\bar{i}=(0,1),\qquad
\bar{j}=(1,0),\qquad
\bar{k}=(1,1)
$$
which is nothing but the Albuquerque-Majid grading.
However, the bilinear map~$\langle\,,\,\rangle$
in no longer the scalar product.
(It has to be replaced by the determinant of the $2\times2$-matrix
formed by the degrees of the elements.)
The elements $i$ and $j$ are odd in this grading while $k$ is even.
We think that it is more natural if these elements have the same parity.

\subsection*{The classification.}

We are interested in \textit{simple} $\G$-commutative algebras.
Recall that an algebra is called simple if it has
no proper (two-sided) ideal.
In the usual commutative associative case, a simple algebra
(over $\R$ or $\C$) is necessarily a division algebra.
The associative division algebras are classified by
(a particular case of) the classical Frobenius theorem.
Classification of simple $\G$-commutative algebras can
therefore be understood as an analog of the Frobenius theorem.

We obtain a classification
of simple associative $\G$-graded commutative algebras over $\R$ and $\C$.
The following theorem is the main result of this paper.

\begin{thm}
\label{CliffC}
Every  finite-dimensional simple associative $\G$-commutative
algebra over $\C$ or over $\R$ is isomorphic to a Clifford algebra. 
\end{thm}

\noindent
The well-known classification of simple Clifford algebras (cf. \cite{Por})
readily gives a complete list:
\begin{enumerate}
\item
The algebras $\Cl_{2m}(\C)$ 
($\cong\M_{2^m}(\C)$)
are the only simple associative
$\G$-commuta\-tive algebras over $\C$.

\item
The real Clifford algebras $\Cl_{p,q}$ with $p-q\not=4k+1$
and the algebras $\Cl_{2m}(\C)$ viewed as algebras over $\R$
are the only real simple associative
$\G$-commuta\-tive algebras.
\end{enumerate}

\noindent
Note that the Clifford algebras $\Cl_{2m+1}(\C)$ and $\Cl_{p,q}$ with $p-q=4k+1$
are of course not simple; however these algebras are graded-simple.

As a consequence of the above classification, we have
a quite amazing statement.

\begin{cor}
\label{CliffCor}
The associative algebra $\M_n$ of $n\times{}n$-matrices
over $\R$ or~$\C$,
can be realized as $\G$-commutative algebra
for some $\G$, if and only if $n=2^m$. 
\end{cor}

Note that gradings on the algebras of matrices and more generally
on associative algebras,
but without the commutativity assumption, is an important subject,
see \cite{Bah,Bah1,Eld1} and references therein.

\subsection*{Universality of $\left(\Z_2\right)^n$-grading.}
It turns out that the  $\left(\Z_2\right)^n$-grading is the most general,
and $\langle\,,\,\rangle$ can always be reduced to the
scalar product.

\begin{thm}
\label{RedThm}
(i)
If the abelian group $\G$ is finitely generated, then
for an arbitrary $\G$-commutative algebra $\A$, there exists $n$ such that $\A$ is
$\left(\Z_2\right)^n$-commutative.

(ii)
The bilinear map~$\langle\,,\,\rangle$
can be chosen as the usual scalar product.
\end{thm}

\noindent
This theorem is proved in Section \ref{RedSect}.
This is the main tool for our classification of simple associative
$\G$-commutative algebras.

\subsection*{Non-associative extensions of Clifford algebras.}

Similarly to the grading of the quaternion algebra
(\ref{DegH}), all the elements of Clifford algebras 
in our grading correspond only to the
even elements of $\left(\Z_2\right)^{n+1}$.
It is therefore interesting to ask the following question.
{\it Given a Clifford algebra, is there a natural
larger algebra that contains  this Clifford algebra as an even part?}

We will show (see Corollary \ref{EvenRrop}) that such an extension
of a Clifford algebra cannot be associative.
In Section \ref{NonAsSect}, we construct 
simple $\left(\Z_2\right)^{n+1}$-commutative
\textit{non-associative} algebras $\A$ such that
$\A^0=\Cl_n$.
We require one additional condition: existence of at least
one \textit{odd derivation}.
This means, the even part $\A_0$ and the odd part $\A_1$
are not separated and can be ``mixed up''.

We have a complete classification only in the simplest case
of $\Z_2$-commutative algebras.
We obtain exactly two 3-dimensional algebras.
One of these algebras has the basis $\{\e;a,b\}$,
where $\bar{\e}=0$ and $\bar{a}=\bar{b}=1$,
satisfying the relations
\begin{equation}
\label{aslA}
\begin{array}{l}
\e\,{}\e=\e\\[6pt]
\e\,{}a=\half\,a,
\quad
\e\,{}b=\half\,b,\\[6pt]
a\,{}b=\e.
\end{array}
\end{equation}
This algebra is called
\textit{tiny Kaplansky superalgebra}, 
see \cite{McC} and also \cite{BE} and denoted~$K_3$.
It was rediscovered in~\cite{Ovs} (under the name $\asl_2$)
and further studied in \cite{MG}.
The corresponding algebra of derivations is the simple Lie superalgebra
$\osp(1|2)$.
We believe that the natural extension of Clifford algebras
with non-trivial odd part are the algebras:
$$
\Cl_n\otimes_\C{}K_3,
\qquad
\Cl_{p,q}\otimes_\R{}K_3
$$
in the complex and in the real case, respectively.

We will however construct another series of extensions
of Clifford algebras, that have different properties.
For instance, these algebras have the unit element.

We hope that some of the new algebras constructed in this paper
may be of interest for mathematical physics.

\bigskip

\noindent \textbf{Acknowledgments}.
The main part of this work was done at the
Mathematisches Forschungsinstitut Oberwolfach (MFO) during a 
\textit{Research in Pairs} stay from April~5 to April 18 2009.
We are grateful to MFO for hospitality.
We are pleased to thank Yu. Bahturin, A. Elduque and D. Leites
for enlightening discussions.


\section{Preliminary results}\label{Prems}

We start with simple results and observations.


\subsection{Clifford algebra is indeed commutative}\label{CliffCom}

A real Clifford algebra $\Cl_{p,q}$
is an associative algebra with unit $\e$
and $n=p+q$ generators
$\a_1,\ldots,\a_n$ subject to the relations
$$
\a_i\a_j=-\a_j\a_i,
\quad{}i\not=j.
\qquad
\a_i^2=\left\{
\begin{array}{rl}
\e,&1\leq{}i\leq{}p\\[4pt]
-\e,&p<i\leq{}n.
\end{array}
\right.
$$
The complex Clifford algebra $\Cl_n=\Cl_{p,q}\otimes\C$ 
can be defined by the same formul\ae,
but one can always choose the generators in such a way that $\a_i^2=\e$
for all $i$.

Let us show that $n$-generated Clifford algebras are
$\left(\Z_2\right)^{n+1}$-commutative.
The construction is the same in the real and complex cases.
We assign the following degree to every basis element.
\begin{equation}
\label{DegCl}
\begin{array}{rcl}
\overline{\a_1}&=&(1,0,0,\ldots,0,1),\\[4pt]
\overline{\a_2}&=&(0,1,0,\ldots,0,1),\\[4pt]
\cdots&&\\[4pt]
\overline{\a_n}&=&(0,0,0\ldots,1,1).
\end{array}
\end{equation}
One then has $\langle\a_i,\a_j\rangle=1$ so that
the anticommuting generators $\a_i$ and $\a_j$
become commuting in the $\left(\Z_2\right)^{n+1}$-grading sense.
Furthermore, two monomials
$\a_{i_1}\cdots\a_{i_k}$ and $\a_{j_1}\cdots\a_{j_\ell}$
commute if and only if either $k$ or $\ell$ is even.
It worth noticing that every monomial is homogeneous and
there is a one-to-one correspondence between
the monomial basis of the Clifford algebra and the even
elements of $\left(\Z_2\right)^{n+1}$.

We notice that the degrees of elements in (\ref{DegCl}) are purely even.
We will show in Section \ref{ComDivZ} 
that this is not a coincidence.

\begin{rem}
{\rm
A $\left(\Z_2\right)^n$-graded commutative structure of the Clifford algebras
was defined in \cite{AM1}:
the degree of a generator $\a_i$ is the $n$-vector
$(0,\ldots,0,1,0,\ldots,0)$, where 1 stays at the $i$-th position.
As in the case of the quaternion algebra, our $\left(\Z_2\right)^{n+1}$-grading
is equivalent to that of Albuquerque-Majid.
However, the symmetric bilinear map~of \cite{AM1} is different of ours,
namely
$$
\b(\g,\g')=\sum_i\g_i\,\g'_i+
\Big(\sum_i\g_i\Big)\,\Big(\sum_j\g'_j\Big),
$$
instead of the scalar product.
We will show in the next section that
an $n$-generated Clifford algebra cannot be realized as
$\left(\Z_2\right)^{k}$-commutative algebra with $k<n+1$,
provided the bilinear map~$\langle\,,\,\rangle$ is the scalar product.

Let us also mention that the classical $\Z_2$-grading of the Clifford algebras,
see~\cite{ABS}
can be obtained from our grading by projection to the last $\Z_2$-component
in~$\left(\Z_2\right)^{n+1}$.
Indeed, degree 1 is then assigned to every generator.
}
\end{rem}

\subsection{Commutativity, simplicity and zero divisors}\label{ComDivZ}

Unless we specify the ground field,
the results of this section hold over $\C$ or $\R$.
A simple associative commutative
algebra is a division algebra.
Indeed, for every $a\in\A$, the set 
\begin{equation}
\label{TwiSid}
\{b\in \A \,|\, ab=0\}
\end{equation}
is a (two-sided) ideal.

A simple $\G$-commutative associative algebra
can have zero divisors.
Consider for example the complexified quaternion algebra
$\bbH\otimes\C\cong\M_2(\C)$.
This is of course a simple $\left(\Z_2\right)^3$-commutative algebra.
It has elements $\a$ such that $\a^2=\e$.
One then obviously obtains:
$\left(\e+\a\right)\left(\e-\a\right)=0$.

The following lemma shows however that
there is a property of simple $\G$-commu\-tative associative algebras
similar to the usual commutative case.

\begin{lem}
\label{lem1}
Every homogeneous element
of a simple associative $\G$-commutative algebra is
not a left or right zero divisor.
\end{lem}

\begin{proof}
If $a$ is a homogeneous element then 
for every $c\in\A$, there exists $\widetilde{c}\in\A$ such that
$ac=\widetilde{c}a$.
Indeed, writing $c$ as a sum of homogeneous terms one obtains:
$$
c=\sum_{\g\in\G}c_\g,
\qquad 
\widetilde{c}=\sum_{\g\in\G}(-1)^{\bar{a}\bar{\g}}\,c_\g.
$$
It follows that the set (\ref{TwiSid}) is again a two-sided ideal.
By simplicity of $\A$ this ideal has to be trivial.
\end{proof}

\begin{defi}
{\rm
Let us introduce the \textit{parity function}
\begin{equation}
\label{Parity}
p(a)=\left\langle
\bar{a},\,\bar{a}
\right\rangle,
\end{equation}
where $a,b\in\A$ are homogeneous.
A $\G$-commutative algebra is then split as a vector space into
$\A=\A^0\oplus\A^1$, where $\A^0$ and $\A^1$ are generated
by even and odd elements, respectively.
}
\end{defi}

One thus obtains a $\Z_2$-grading on $\A$,
that however, does not mean $\Z_2$-commuta\-tivity of~$\A$.
In particular, the even subspace $\A^0$ is a subalgebra of $\A$
which is not commutative in general.

Lemma \ref{lem1} has several important corollaries.

\begin{cor}
\label{EvenRrop}
Every simple associative $\G$-commutative algebra $\A$ is
even, that is, $\A=\A^0$.
\end{cor}

\noindent
Indeed, an odd element $a\in\A^1$ anticommutes with itself, so that
$a^2=0$. 
But this is impossible by Lemma~\ref{lem1}.
It follows that the odd part of $\A$
is trivial, that is, $\A^1=\{0\}$.

Denote by $\A_0$ the subalgebra of $\A$ consisting of 
homogeneous elements
of degree~$0\in\G$.
Another corollary of Lemma \ref{lem1} is as follows.

\begin{cor}
\label{FroCol}
Let $\A$ be a simple associative $\G$-commutative algebra, then:

(i)
In the complex case, one has $\A_0=\C$;

(ii)
In the real case, $\A_0=\R$ or $\C$ (viewed as an $\R$-algebra).
\end{cor}

\noindent
Indeed, the space $\A_0$ is a commutative associative
division algebra. 
By the classical Frobenius theorem, $\A_0=\R$ or $\C$
in the real case and $\A_0=\C$ in the complex case.

A strengthened version of the above corollary in the complex case
is as follows.

\begin{cor}
\label{GenCol}
If $\A$ is a simple complex associative 
$\left(\Z_2\right)^n$-commutative algebra
and~$\A_\g$ is a non-zero homogeneous component of $\A$, then
$\dim\A_\g=1$.
Furthermore,
there exists $\a\in\A_\g$ such that $\a^2=\e$.
\end{cor}

\noindent
Indeed, for every $a,b\in\A_\g$, the product
$ab$ belongs to $\A_0$ since $\G=\left(\Z_2\right)^n$.
This product is different from zero (Lemma~\ref{lem1}).
It follows that every homogeneous component is at most
one-dimensional.

As a first application of the above statements,
we can now easily see that a complex Clifford algebra with $n$ generators
cannot be realized as a $\left(\Z_2\right)^k$-commutative algebra
with $k<n+1$.
Indeed, the dimension of the Clifford algebra is equal to $2^n$
and there are exactly $2^n$ even elements in $\left(\Z_2\right)^{n+1}$.
Our claim then follows from Corollary \ref{GenCol}.

\section{Reducing the abelian group}\label{RedSect}

In this section we prove Theorem \ref{RedThm}.

\subsection{Universality of $\left(\Z_2\right)^n$}\label{RedUniSect}

Let $\G$ be a finitely generated abelian group.
By the fundamental theorem of finitely generated abelian groups,
one can write $\G$ as a direct product
$$
\G=\Z^n\times
\Z_{n_1}
\times\cdots\times
\Z_{n_m},
$$
where $n_i=p_i^{k_i}$ and where $p_1,\ldots,p_m$ are
not necessarily distinct prime numbers.
Assume that there is a non-trivial bilinear map
$\langle\,,\,\rangle:\G\times\G\to\Z_2$.

In the case where $\G=\G'\times\G''$,
if $\langle{}y,x\rangle=0$ for all $x\in\G'$ and $y\in\G$, then
the component $\G'$ is not significative in the $\G$-grading of $\A$,
i.e., $\A$ is $\G''$-commutative.

Let us choose one of the components,
$\G'=\Z$ or $\G'=\Z_{n_i}$.
We can assume that there exists an element $x\in\G$,
such that the map
$$
\varphi_x:\G'\to\Z_2,
$$
defined on $\G'$ by $\langle\,,x\rangle$ defines a
non-trivial group homomorphism.

In the case where $\G'=\Z$, the only non-trivial homomorphism is
defined by $\varphi(0)=0$ and $\varphi(1)=1$.
Therefore, one can replace the component $\G'=\Z$ of $\G$ by $\G'=\Z_2$.
The algebra $\A$ remains $\G$-commutative.

In the case where $\G'=\Z_{n_i}$, a non-trivial homomorphism exists if and only if
$n_i$ is even.
We may then assume that $n_i=2^{k_i}$.
However, in this case, again the only non-trivial homomorphism is
defined by $\varphi(0)=0$ and $\varphi(1)=1$ so that
one replaces the component $\G'=\Z_{2^{k_i}}$ of $\G$ by $\G'=\Z_2$ 
without loss of $\G$-commutativity.

The first part of Theorem \ref{RedThm} is proved.

\subsection{Universality of the scalar product}\label{RedScalSect}

Let us now prove the second part of Theorem \ref{RedThm},
namely that the standard scalar product is the only relevant bilinear map
from $\G$ to $\Z_2$.

Given an arbitrary bilinear symmetric form 
$\beta:\left(\Z_2\right)^n\times\left(\Z_2\right)^n\to \Z_2$, 
we will show that there exists an integer $N\leq 2n$
and an abelian group homomorphism 
$$
\sigma:\left(\Z_2\right)^n\to\left(\Z_2\right)^N
$$
such that 
$$
\langle \sigma(x),\sigma(y)\rangle =
\beta(x,y),
$$
for all $x,y\in\left(\Z_2\right)^n$ and
where $ \langle \,,\,\rangle$ is
the standard scalar product on $\left(\Z_2\right)^N\times\left(\Z_2\right)^N$.

Consider the standard basic  elements of $\Z_2^n$:
$$
\e_i=(0, \cdots, 0,1, 0,\cdots,0),
$$
where all the entries are zero except the $i$-th entry that is equal to one.
The form~$\beta$ is completely determined by the numbers
$$
\beta_{i,j}:=\beta(\e_i,\e_j),
\qquad
1\leq i,j\leq n.
$$
We first construct a family of vectors $\sigma_i$ in 
$\left(\Z_2\right)^n$ that have the following property
$$
\langle \sigma_i, \sigma_j\rangle=\beta_{i,j},
\qquad
\hbox{for all}
\quad
i\not=j.
$$
The explicit formula for $\sigma_i$ is: 
\begin{equation*}
\begin{array}{rccccccccccl}
\sigma_1&=&(&1&, &0&,&0&,& \cdots& 0&)\\[6pt]
\sigma_2&=&(&\beta_{12}&, &1&, &0&,&\cdots&0&) \\[6pt]
\sigma_3&=&(&\beta_{13}&,& \beta_{23}-\beta_{12}\beta_{13} &,&1&,&\cdots&0& )\\[6pt]
&\vdots&&\vdots&&\vdots\\[6pt]
\sigma_n&=&(&\beta_{1n}&, &  \beta_{2n}-\beta_{12}\beta_{1n} &, &\cdots&,&\cdots&1&) 
\end{array} 
\end{equation*}
This construction does not guarantee that 
$\langle \sigma_i, \sigma_i\rangle=\beta_{i,i}$. 
However, by adding at most $n$ columns, we can obviously
satisfy these identities as well.

Theorem \ref{RedThm} is proved.

\medskip

Since now on, we will assume that $\G=\left(\Z_2\right)^n$;
this is the only group relevant for the notion of $\G$-commutative algebra.
We will also assume that the bilinear map~$\langle\,,\,\rangle$ 
is the usual scalar product.

\section{Completing the classification}

In this section we prove Theorem \ref{CliffC}.

\subsection{The complex case}

Let $\A$ be a simple associative $\left(\Z_2\right)^n$-commutative algebra over $\C$.
We will be considering minimal sets of homogeneous generators of $\A$.
By definition, two homogeneous elements $\a,\b\in\A$ either commute or anticommute.
We thus can organize the generators in two sets
$$
\{\alpha_1,\cdots, \alpha_p\}\cup \{\beta_1,\cdots, \beta_q\}
$$
where the subset $\{\alpha_1,\cdots, \alpha_p\}$ is the biggest subset 
of pairwise anticommutative generators: $\a_i\a_j=-\a_j\a_i$, while each
generator $\b_i$ commutes with at least one generator $\a_j$.
We can assume $p\geq 2$ otherwise the algebra $\A$ is commutative and so it is
$\C$ itself (by Frobenius theorem).

Among the minimal sets of homogeneous generators we choose one with the greatest $p$. 
If $q=0$ then $\A$ is exactly $\Cl_p$. 
Suppose that $q>0$.

\begin{lem}
\label{lem1alpha}
The generators $\{\beta_1,\cdots, \beta_q\}$ can be chosen in such
a way that every $\beta_i$ commutes with exactly one element in  
$\{\alpha_1,\cdots, \alpha_p\}$. 
\end{lem}

\begin{proof}
Suppose $\beta_i$ anticommutes with $\alpha_1, \cdots, \alpha _s$
and commutes with $\alpha_{s+1}, \cdots, \alpha_{p}$. 
Changing $\beta_i$ to $\tilde{\beta}_i:=\alpha_{s+1}\alpha_{s+2}\beta_i$,
one obtains a new generator that anticommutes with $\alpha_1, \cdots, \alpha _{s+2}$ 
and commutes with $\alpha_{s+3}, \cdots, \alpha_{p}$. 
Repeating this procedure, we can change $\beta_i$ 
to a new generator that commutes with at most one of the $\alpha_j$'s.  

By the assumption of maximality of $p$, the generator $\beta_i$ has to commute 
with at least one of the $\alpha_j$'s.  
\end{proof}

\begin{lem}
\label{lembeta}
The generators $\{\beta_1,\cdots, \beta_q\}$ can be chosen 
commuting with each other and with the same generator $\alpha_p$. 
\end{lem}

\begin{proof}
The construction of such a set of generators will be obtained by induction.
Consider $\beta_1$ and $\beta_2$. 
By Lemma \ref{lem1alpha}, they both commutes with one of the 
$\alpha$-generators. 
Let $\beta_1$ commutes with $\alpha_{p_1}$ and $\beta_2$ commutes with $\alpha_{p_2}$.
There are four cases:
\begin{itemize}
\item[(1)] $\beta_1$ commutes with $\beta_2$ and $\alpha_{p_1}=\alpha_{p_2}$,
\item[(2)] $\beta_1$ anticommutes with $\beta_2$ and $\alpha_{p_1}=\alpha_{p_2}$,
\item[(3)] $\beta_1$ commutes with $\beta_2$ and $\alpha_{p_1}\not=\alpha_{p_2}$,
\item[(4)] $\beta_1$ anticommutes with $\beta_2$ and $\alpha_{p_1}\not=\alpha_{p_2}$,
\end{itemize}

Let us show that we can always choose the generators
in such a way that the case~(1) holds.
Indeed, in the case (2), we get a bigger set  
$$
\{\alpha_1,\cdots, \alpha_p\}\setminus\{\alpha_{p_1}\}\cup\{\beta_1,\beta_2\}
$$ 
of pairwise anticommutative generators. 
Therefore, case (2) is not possible.

In the case (3), we replace $\beta_2$ by $\tilde{\beta_2}:=\beta_1\beta_2\alpha_{p_2}$. 
The new generator $\tilde{\beta_2}$ anticommutes with all the $\alpha_j$'s, 
where $j\not=p_1$, commutes with $\alpha_{p_1}$ and anticommutes with $\beta_1$. 
Thus, we got back to case (2), that is a contradiction.

In the case (4), we replace $\beta_1$ by $\tilde{\beta_2}:=\beta_1\beta_2\alpha_{p_2}$. 
The generator $\tilde{\beta_2}$ anticommutes with all the $\alpha_k$, $k\not=p_1$, 
commutes with $\alpha_{p_1}$ and commutes with $\beta_1$.

Thus, we obtain a set of generators satisfying the case (1).
This provides the base of induction.

Suppose that $\ell$ generators $\{\beta_1,\cdots, \beta_ \ell\}$
pairwise commute and commute with the same generator $\alpha_p$. 
Consider an extra generator $\beta_{\ell+1}$
that commutes with $\a_{p'}$. 
Replacing $\beta_{\ell +1}$ by 
$\widetilde{\beta}_{\ell +1}:=\alpha_{p}\alpha_{p'}\beta_{\ell +1}$,
one obtains a generator commuting with $\a_p$ and
anticommuting with all the $\a_j$'s, for $j\not=p$. 
If there is a $\beta_i$, for $i\leq \ell$ such that $\beta_i$ and $\widetilde{\beta}_{\ell+1}$ 
anticommute, then we get a bigger set
of pairwise anticommuting generators:
$$
\{\alpha_1,\cdots, \alpha_{p-1},\beta_i,\widetilde{\beta}_{\ell+1}\}
$$
that contradicts the maximality of $p$.

In conclusion, we have constructed a set $\{\beta_1,\cdots, \beta_{\ell+1}\}$
of pairwise commuting elements that commute with the same element $\alpha_p$
and anticommute with the rest of $\a_j$'s.
\end{proof}

We now choose a set of generators 
$\{\alpha_1,\cdots, \alpha_p,\,\beta_1,\cdots, \beta_q\}$
as in Lemma \ref{lembeta}. 
In addition,  we can normalize the generators by $\alpha_i^2=\beta_i^2=\e$ for all $i$.
Indeed, by Theorem \ref{RedThm} we assume $\G=\left(\Z_2\right)^n$ and
then use Corollary \ref{GenCol}.

\begin{lem}
\label{LastLeM}
If the number $q$ of commuting generators is greater than zero, then
the algebra $\A$ cannot be simple.
\end{lem}

\begin{proof}
First, the space $\left(\a_p+\b_1\right)\A$ is an ideal of $\A$. 
Indeed, this space is clearly a left ideal. 
It is also a right ideal because  $\alpha_p$ and $\beta_1$
commute or anticommute with any generator of $\A$ simultaneously.
Therefore, 
$$
\left(\a_p+\b_1\right)c=\widetilde{c}\left(\a_p+\b_1\right).
$$

It remains to show that $\left(\a_p+\b_1\right)\A$ is a proper ideal.
If $\A=(\alpha_p+\beta_1)\A$ then we can write 
$$
\e=\left(\a_p+\b_1\right)a,
$$
for some $a\in \A$.
Multiplying this equality by $(\alpha_p-\beta_1)$, we get
$$
(\alpha_p-\beta_1)=(\alpha_p-\beta_1)\left(\a_p+\b_1\right)a.
$$
But we have 
$$
(\alpha_p-\beta_1)(\alpha_p+\beta_1)=\alpha_p^2-\beta_1^2=
\e-\e=0
$$
so we deduce 
$\alpha_p-\beta_1=0.$
This is not possible because the set of generators is chosen minimal
so that $\a_p$ and $\b_1$ are different.
\end{proof}

Combining Lemmas \ref{lem1alpha}--\ref{LastLeM}, we conclude
that $\A$ is generated by homogeneous anticommuting generators.
Theorem \ref{CliffC} is proved in the complex case.

\subsection{The real case}

In the case of a real simple associative $\G$-commutative algebra $\A$,
the component of degree $0\in\G$ can be one- or two-dimensional:
 $$\A_0=\R \; \text{ or } \C,$$
 see Corollary \ref{FroCol}.
 
 \medskip
 
\textbullet Case $\A_0=\R$.

\medskip
 
 We proceed in a similar way as in the complex case. 
 Lemma \ref{lembeta} still holds.
 The proof of Lemma \ref{LastLeM} is however false
 because we can not assume that 
 $\alpha_p^2=\beta_1^2=\e$. 
 We can only assume  $\alpha_i^2=\pm\e$ and $\beta_i^2=\pm\e$ for all $i$
 (we again use Corollary \ref{GenCol}).
 
 Choose a system of homogeneous generators 
$$
\{\alpha_1,\cdots, \alpha_p\}\cup \{\beta_1,\cdots, \beta_q\}
$$
such that:

\goodbreak

\begin{enumerate}

\item
The system is minimal;

\item
The generators $\{\alpha_1,\cdots, \alpha_p\}$ pairwise anticommute
and the number $p$ of anticommuting generators
is maximal (within the minimal sets of homogeneous generators);

\item
The elements $ \{\beta_1,\cdots, \beta_q\}$ pairwise commute,
they also commute with $\alpha_p$
and anticommute with $\alpha_1,\ldots,\a_{p-1}$.

\end{enumerate}

The existence of such a system of generators is guaranteed by Lemma \ref{lembeta}.

\begin{lem}
\label{PlohLem}
The number $q$ of commuting generators is zero.
\end{lem}

\begin{proof}
If $\alpha_p^2=\beta_1^2$ then, as in the complex case,
the element $\alpha_p+\beta_1$ generates a proper ideal.
We will assume that 
$$\alpha_p^2=-\e,
\qquad
\beta_1^2=\e.
$$

If $q\geq 2$ then among $\alpha_p, \beta_1,\cdots, \beta_q$ 
at least two elements have same square.
The sum of these two elements again generates a proper ideal.
We therefore have the last possibility: $q=1$.
  
  If $p$ is even, then we replace $\beta_1$ by 
  $$
  \widetilde{\beta}_1=\beta_1\alpha_1\cdots \alpha_{p-1}.$$
  It is then easy to check that $\widetilde{\beta}_1$ anticommutes with all the $\alpha_i$'s.
  Therefore we have obtained a system of $p+1$ pairwise anticommuting generators
  $\{\alpha_1,\cdots, \alpha_p,\,\widetilde{\beta}_1\}$.
  This is a contradiction with the maximality of $p$.
  
  Finally, we assume that $p$ is odd. Let us introduce the elements
  \begin{eqnarray*}
  \widetilde{\alpha}&=&\alpha_1\cdots \alpha_p,\\
  \widetilde{\beta}&=&\alpha_1\cdots \alpha_{p-1}\beta_1.
  \end{eqnarray*}
  It is easy to check that $ \widetilde{\alpha}$ and $ \widetilde{\beta}$
  both commute with all the generators $\alpha_i$'s and with $\beta_1$.
  Furthermore,
  \begin{eqnarray*}
  \widetilde{\alpha}^2&=&(\alpha_1\cdots \alpha_{p-1})^2\, \alpha_p^2\\
  \widetilde{\beta}^2&=&(\alpha_1\cdots \alpha_{p-1})^2\,\beta_1^2\, .
  \end{eqnarray*}  
  Since we are in the case of $\alpha_p^2\not=\beta_1^2$, we have 
  $ \widetilde{\alpha}^2\not=
  \widetilde{\beta}^2$. 
  Assume without loss of generality that $ \widetilde{\alpha}^2=\e$.
  The space
  $\left(\e+ \widetilde{\alpha}\right)\A$
  is a (two-sided) ideal of $\A$ because $\e$ and $ \widetilde{\alpha}$ 
  commute with all the generators of $\A$.
  In addition,
  $$
  (\e- \widetilde{\alpha})(\e+\widetilde{\alpha})=
  \e- \widetilde{\alpha}^2=0
  $$
  implies that $\e$ does not belong to $(\e+ \widetilde{\alpha})\,\A$
  so that the ideal is proper.
  \end{proof}
  
In conclusion, the existence of an element $\beta_1$ leads to contradictions.
Consequently, the algebra $\A$ is generated by pairwise anticommutative generators.
Therefore, $\A$ is isomorphic to a real Clifford algebra.
Theorem \ref{CliffC} is proved in the case  $\A_0=\R$.

 \medskip
 
\textbullet 
Case $\A_0=\C$.

\medskip
 
Since the zero component of $\A$ is two-dimensional,
one has the following statement.

\begin{lem}
\label{TwoComL}
Every non-trivial homogeneous component $\A_\g$ is 2-dimensional
and contains elements $\a_+$ and $\a_-$ such that
$$
\left(\a_+\right)^2=\e,
\qquad
\left(\a_-\right)^2=-\e.
$$
\end{lem}

\begin{proof}
Denote the basis of $\A_0$ by $\left\{\e,i\right\}$.
A non-zero component $\A_\g$ is at least two-dimensional
since for every $\a\in\A_\g$, the element $i\,\a$ is linearly
independent with~$\a$.
Indeed, if $\l\,\a+\mu\,i\a=0$, with $\l,\mu\in\R$, then
$(\l+\mu\,i)\a=0$, thus, by Lemma \ref{lem1}, $\l+\mu\,i=0$,
so that $\l=\mu=0$.
Furthermore, combining $\a$ and $i\,\a$, one easily finds
two elements $\a_+,\a_-\in\A_\g$ such that
$\left(\a_+\right)^2=\e$ and $\left(\a_-\right)^2=-\e$.

It remains to prove that $\dim\A_\g=2$.
Suppose $\dim\A_\g\geq3$.
That there exists $\b\in\A_\g$ linearly independent with
$\a$ and~$i\,\a$.
Since $\a\b\in\A_0$, there is a linear combination
$\l\,\e+\m\,i+\nu\,\a\b=0$.
Multiplying this equation by $\a$ 
(and assuming without loss of generality $\a^2=\e$), one has:
$$
\l\,\a+\mu\,i\,\a+\nu\,\b=0.
$$
Hence a contradiction.
\end{proof}

The algebra $\A$ is therefore a $\C$-algebra.
The end of the proof of Theorem \ref{CliffC} in the case where
$\A_0=\C$ is exactly the same as in the complex case.

Theorem \ref{CliffC} is now completely proved.
  
\section{Non-associative extensions of the Clifford algebras} \label{NonAsSect}

In this section, we construct simple
$\left(\Z_2\right)^n$-commutative algebras extending the Clifford algebras.
The algebras we construct contain the Clifford algebras as even parts.
According to Corollary \ref{EvenRrop}, such algebras cannot be associative.

Recall that, if $\A$ is a $\G$-commutative algebra, then the space
$\End(\A)$ is naturally $\G$-graded.
A homogeneous linear map $T\in\End(\A)$ is called a derivation of $\A$ if
for all homogeneous $a,b\in\A$ one has
\begin{equation}
\label{InvGrad}
T
\left(
ab
\right)=
T\left(a\right)\,{}b
+
(-1)^{\langle\overline{T},\bar{a}\rangle}\,
a\,T\left(b\right).
\end{equation}
This formula then extends by linearity for arbitrary $T$ and $a,b$.
The space $\mathrm{Der}(\A)$ of all derivations of $\A$ is a $\G$-graded Lie algebra.

We will restrict our considerations to the case of
$\left(\Z_2\right)^n$-commutative algebras that have
non-trivial odd derivations.
This means we assume that there exists a derivation
$T$ exchanging $\A^0$ and $A^1$.
We think that this assumption is quite natural since
in this case the two parts of $\A$ are not separated from each other.

\subsection{Classification in the $\Z_2$-graded case}

Let us start with the simplest case of
$\Z_2$-commutative algebras.
The following statement provides a classification of such algebras.

\begin{prop}
\label{Z2Prop}
There exist exactly two
simple $\Z_2$-commutative algebras $\A$ with the 
following two properties:
\begin{itemize}
\item
The even part $\A^0=\K$, where $\K=\C$ or $\R$;
\item
There exists a non trivial odd derivation $T$ of $\A$.
\end{itemize}
\end{prop}

\begin{proof}
Denote by $\e$ the unit element of $\A^0$.
There are two different possibilities.

\begin{enumerate}

\item
$T(\e)=0$, i.e., $T\,|_{\A^0}=0$;

\item

$T(\e)=a\not=0$.

\end{enumerate}

In the case (1), there exists an odd element $a$ such that $T(a)=\e$. 
One then has 
$$
\A^1=\K\,a\oplus\ker T.
$$ 
Moreover,
$
\e\,b=0,
$
for all $b\in \ker T$.
Indeed, $0=T(ab)=T(a)\,b=\e\,b$.
One has
$
\e\,a=\lambda\,a+b,
$
 for some $\lambda \in \K$ and $b\in \ker T$.
Applying (\ref{InvGrad}) one gets
$$
T(\e a)=T(\e)\,a+\e\,T(a)=\e.
$$
On the other hand,
$$
T(\e a)=T(\lambda a+b)=\lambda\,\e.
$$
Therefore, $\lambda =1$ and $\e (a+b)=a+b$.
Replacing $a$ by $a+b$, one gets
$\e\,a =a.$

 It follows that:
 a) if the space $\ker T$ is non-trivial, then
 the space spanned by $\e$ and $a$ is a proper ideal;
 b) if $\ker T$ is trivial, then $ \K\,a$ is a proper ideal. 
This is a contradiction with the simplicity
assumption, therefore the case (1) cannot occur.\\
 
 Consider the case (2) where $T(\e)=a\not=0$.
 From
 $
 T(\e)=T(\e \e)=2\e\,T(\e)
 $
 one deduces
$
 \e\,a= \frac{1}{2} a.
$
 Applying $T$, one obtains
$
 T( a)= 0.
$
 
 \begin{lem}
 \label{NonDeGL}
The product in $\A$ restricted to $\A^1$, that is,
$\A^1\times \A^1\rightarrow \K\,\e$
is a non-degenerate bilinear skewsymmetric form. 
 \end{lem}
 
 \begin{proof}
 Suppose there exists $c\in \A^1$ such that $cb=0$ for all $b\in \A^1$. 
 We will show that we also have $\e c=0$.
 Consequently the element
 $c$ generates a two-sided ideal that contradicts the assumption of simplicity.
 
One has $T(c)=\mu\, \e$ for some $\mu$. 
On the one hand,  
$$
\textstyle
T(c)\,a=\mu\,\e a =\frac{\mu}{2}\,a.
$$
On the other hand, 
$$
T(c)\,a=T(ca)-c\,T(a)=0.
$$ 
Therefore one obtains $\mu =0$, i.e.,
$T(c)=0.$

We can find $b\in \A^1$ such that 
$ba=\e.$
Indeed, if $ba=0$ for all $b\in \A^1$ then $\C\,a$ is an ideal.
From $\e\,a=\half\,a$ and
$$
T(ba)=T(\e)=a
=T(b)\,a-b\,T(a)=T(b)\,a
$$
one deduces
$T(b)=2\e.$
Now, from
$$
T(bc) =0
=T(b)\,c-b\,T(c)=2\,\e c
$$
one gets 
$\e c=0.$
  \end{proof}

Let us denote the bilinear form from Lemma \ref{NonDeGL}
by $\om$.

\begin{lem}
The space $\A^1$ is 2-dimensional.
\end{lem}

\begin{proof}
The dimension of $\A^1$ is even and there exists a basis 
$\{a_1, \cdots, a_m,b_1, \cdots, b_m\}$ such that 
$a_i\,b_j=\delta_{i,j}\,\e.$
Moreover, we can assume $a_1=a$.

For all $b\in \ker T$ one has 
$
T(ab)=T(a)\,b-a\,T(b)=0,
$
since we have proved $T(a)=0$.
So necessarily $ab=0$.
This implies:
$
\ker T \subset\left(\K\,a\right)^{\perp\om}.
$
Since the dimensions of these spaces are the same we have
$$
\ker T = 
\left(\K\,a\right)^{\perp\om}
=\left<a_1, \cdots, a_m,\, b_2, \cdots b_m\right>
$$
If $m>1$ then we have two vectors $a_2, b_2$ such that 
$$
T(a_2b_2)=T(a_2)\,b_2-a_2\,T(b_2)=0.
$$
But $a_2b_2=\e$ and $T(\e)\not=0$. 
This is a contradiction.
\end{proof}

So far we have shown that $\A$ is 
$1|2$-dimensional and has a basis $\{\e;a,b\}$ such that
$$
\textstyle
\e\,a =\half\,a,
\qquad
a\,b=\e.
$$
We want now to determine the product $\e\,b$.

In general, 
$$
\e\,b=\l\,a+\mu\,b,
\qquad
T(b)=
\nu\,b.
$$
Applying the derivation $T$ to the expression for $a\,b$, one obtains
$$
a=T(\e)=T(ab)=T(a)\,b-a\,T(b)=-a\,\nu\e,
$$
that implies $\nu=-2$.
Applying $T$ to the above expression for $\e\,b$, one has
$$
-2\mu\,\e=T(\e\,b)=T(\e)\,b+\e\,T(b)=a\,b+\e\,(-2\,\e),
$$
that gives $\mu=\half$.
The parameter $\l$ cannot be found from the above equations.

One obtains the most general formula for the product
of the basis elements:
$$
\begin{array}{l}
\e\,{}\e=\e\\[6pt]
\e\,{}a=\half\,a,
\quad
\e\,{}b=\half\,b+\l\,a,\\[6pt]
a\,{}b=\e.
\end{array}
$$
To complete the proof, one
observes that all the algebras corresponding to $\l\not=0$
are isomorphic to each other but not isomorphic to the algebra
corresponding to $\l=0$.
One therefor has exactly two non-isomorphic algebras.

Proposition \ref{Z2Prop} is proved.
\end{proof}

If $\l=0$, we recognize the algebra $K_3(\K)$
(see formula~\eqref{aslA}).
This algebra is parti\-cularly interesting.
This is the only simple $Z_2$-commutative algebra that have the Lie superalgebra
$\osp(1|2)$ as the algebra of derivations, see~\cite{Ovs} and also \cite{MG}.
The algebra of derivations in the case $\l\not=0$ is of dimension $1|1$.

\subsection{Extended Clifford algebras}

We associate to every Clifford algebra a 
simple $\left(\Z_2\right)^{n+1}$-commutative algebra 
of dimension $2^n|2^{n+1}$.
It is defined by
the tensor product with $K_3$.
More precisely, we define the ``extended Clifford algebras''
\begin{equation}
\label{VotOni}
 \A=\Cl_n\otimes_\C{}K_3(\C),
 \qquad
 \hbox{and}
 \qquad
 \A=\Cl_{p,q}\otimes_\R{}K_3(\R)
\end{equation}
 in the complex and in the real case, respectively.
 
 The even part $\A^0$ of each of these algebras coincides with the
 corresponding Clifford algebra: we identify $\a\otimes\e$ with $\a$.
 The odd part is twice bigger;
 every odd element is of the form
 $$
 x=\a\otimes{}a+\b\otimes{}b,
 $$
 where $\a,\b$ are elements of the Clifford algebra
 and $a,b$ are the basis elements of~$K_3$.
 The $\left(\Z_2\right)^{n+1}$-grading on $\A^1$ is defined by
 $$
 \overline{\a\otimes{}a}=\bar{\a}+(1,1,\ldots,1),
 \qquad
 \overline{\b\otimes{}b}=\bar{\b}+(1,1,\ldots,1).
 $$
 The product in $\A$ is given by
\begin{equation}
\label{ExtClifF}
\begin{array}{rcl}
 \g\left(
 \a\otimes{}a+\b\otimes{}b
 \right)&=&
 \half\left(\g\a\otimes{}a+\g\b\otimes{}b\right),\\[6pt]
 \textstyle
\left(
 \a\otimes{}a+\b\otimes{}b
 \right)\left(
 \a'\otimes{}a+\b'\otimes{}b
 \right)&=&
 \a\b'-\a'\b,
 \end{array}
\end{equation}
where $\g\in\A_0$.
 Note that $\half$ appearing in (\ref{ExtClifF}) is crucial
 (a similar formula without~$\half$ leads to an algebra with no non-trivial
 odd derivations).
   
\begin{ex}
{\rm
Let us describe with more details the
extended algebra of quaternions $\A=\bbH\otimes_\R{}K_3(\R)$.
This algebra is of dimension $4|8$.
The even part $\A^0=\bbH$ is spanned by
$\left\{\e,i,j,k\right\}$;
the odd part $\A^1$ has the basis
$
\left\{
a,a_i,a_j,a_k,\,b,b_i,b_j,b_k
\right\}
$
and the multiplication is given by the following table.

\bigskip
\begin{displaymath}
  \begin{array}{|c|c|c|c|c|c|c|c|c|c|c|}
\cline{3-3}\cline{4-4} \cline{5-5}\cline{6-6} \cline{8-8}\cline{9-9}
\cline{10-10}\cline{11-11}
   \multicolumn{1}{c|}{\cdot} & &a & a_i & a_j & a_k& &b & b_i & b_j & b_k \\
\cline{1-1} \cline{3-3}\cline{4-4} \cline{5-5}\cline{6-6} \cline{8-8}\cline{9-9}
\cline{10-10}\cline{11-11}
\multicolumn{11}{c}{}\\
\cline{1-1} \cline{3-3}\cline{4-4} \cline{5-5}\cline{6-6} \cline{8-8}\cline{9-9}
\cline{10-10}\cline{11-11}
2\e& &a & a_i & a_j & a_k& &b & b_i & b_j & b_k  \\
\cline{1-1} \cline{3-3}\cline{4-4} \cline{5-5}\cline{6-6} \cline{8-8}\cline{9-9}
\cline{10-10}\cline{11-11}
2i& & a_i & -a & a_k & -a_j & &b_i & -b & b_k & -b_j \\
\cline{1-1} \cline{3-3}\cline{4-4} \cline{5-5}\cline{6-6} \cline{8-8}\cline{9-9}
\cline{10-10}\cline{11-11}
2j& & a_j & -a_k & -a & a_i & &b_j & -b_k & -b & b_i  \\
\cline{1-1} \cline{3-3}\cline{4-4} \cline{5-5}\cline{6-6} \cline{8-8}\cline{9-9}
\cline{10-10}\cline{11-11}
2k& & a_k & a_j & -a_i & -a& &b_k & b_j & -b_i & -b \\
\cline{1-1} \cline{3-3}\cline{4-4} \cline{5-5}\cline{6-6} \cline{8-8}\cline{9-9}
\cline{10-10}\cline{11-11}
\multicolumn{11}{c}{}\\
\cline{1-1} \cline{3-3}\cline{4-4} \cline{5-5}\cline{6-6} \cline{8-8}\cline{9-9}
\cline{10-10}\cline{11-11}
 a &&0&0&0&0& &\e& i & j & k \\
 \cline{1-1} \cline{3-3}\cline{4-4} \cline{5-5}\cline{6-6} \cline{8-8}\cline{9-9}
\cline{10-10}\cline{11-11}
 a_i &&0&0&0&0& &i& -\e& k& -j \\
 \cline{1-1} \cline{3-3}\cline{4-4} \cline{5-5}\cline{6-6} \cline{8-8}\cline{9-9}
\cline{10-10}\cline{11-11}
  a_j &&0&0&0&0& &j& -k& -\e& i \\ 
  \cline{1-1} \cline{3-3}\cline{4-4} \cline{5-5}\cline{6-6} \cline{8-8}\cline{9-9}
\cline{10-10}\cline{11-11}
   a_k &&0&0&0&0& &k& j& -i& -\e \\ 
   \cline{1-1} \cline{3-3}\cline{4-4} \cline{5-5}\cline{6-6} \cline{8-8}\cline{9-9}
\cline{10-10}\cline{11-11}
\end{array}
\end{displaymath}

 \bigskip
 \noindent
The table can be completed using the multiplication of quaternions and the graded-commutativity.
For instance, we get
$$
\textstyle
a_i\, i=i\,a_i =-\half\,a, 
\qquad a_j\, i=-i\,a_j =-\half\,a_k,
$$
for the products of even and odd elements.
For the products of odd elements with each other we have:
$$
a\,b=-b\,a=\e,
\qquad
a\,b_i=-b_i\,a=i,
\qquad
a_i\,b_j=b_j\,a_i=k,
$$
etc.
}
\end{ex}

\begin{rem}
{\rm
The extended quaternion algebra $\bbH\otimes_\R{}K_3(\R)$
has an interesting resemblance to the octonion algebra.
It worth mentioning that the octonion algebra itself cannot be realized
as a $\G$-commutative algebra (cf. \cite{MO}).
We cite \cite{Eld} for a complete classification of group gradings
on the octonion algebra without the commutativity requirement.
}
\end{rem}

One can show that the constructed algebras (\ref{VotOni}) satisfy cubic identities.
For instance, the odd elements satisfy the graded Jacobi identity.
This property is not too far from the associativity of Clifford algebras.
We think that these algebras are the only possible extensions of
Clifford algebras satisfying cubic identity.

 The extended Clifford algebras have large algebras of derivations.
 The following statement can be checked by a straightforward calculation.
 
 \begin{prop}
 The algebras of derivations of the extended Clifford algebras are
 the following $\left(\Z_2\right)^{n+1}$-graded Lie algebras:
  $$
  \mathrm{Der}(\A)=\Cl_n\otimes\osp(1|2),
  \qquad
  \hbox{or}
  \qquad
  \mathrm{Der}(\A)=\Cl_{p,q}\otimes\osp(1|2),
  $$ 
  respectively.
  \end{prop}
  
  \noindent
  It is quite remarkable that,
  for every two elements $a,b\in\A$, there exists 
  $T\in\mathrm{Der}(\A)$ such that
  $T(a)=T(b)$.
  We conjecture that the defined algebras are the only simple
  $\left(\Z_2\right)^{n+1}$-commutative algebras satisfying this property.

\subsection{Further examples}

Let us show more examples of simple $\left(\Z_2\right)^{n+1}$-commuta\-tive algebras
that contain the Clifford algebras as even part.
These algebras are $2^n|2^n$-dimensional.
A nice property of of these new algebras is that they have the unit element.
However, their algebra of derivations is too small.

In the simplest case $\A_0=\bbH$, the basis of the algebra is:
$\left\{\e,i,j,k;\,a,a_i,a_j,a_k\right\}$.
The non-trivial odd derivation is as follows:
$$
\begin{array}{rrrrrcllll}
T_{(1,1,1)}:&\e,&i,&j,&k&\longmapsto&0,&a_i,&a_j,&a_k,\\[4pt]
T_{(1,1,1)}:&a,&a_i,&a_j,&a_k,&\longmapsto&\e,&0,&0,&0.
\end{array}
$$
The complete multiplication table is:

\bigskip

\begin{displaymath}
  \begin{array}{|c|c|c|c|c|c|c|c|c|c|c|}
\cline{3-3}\cline{4-4} \cline{5-5}\cline{6-6} \cline{8-8}\cline{9-9}
\cline{10-10}\cline{11-11}
   \multicolumn{1}{c|}{\cdot} & &\e & i & j & k & &a & a_i & a_j & a_k \\
\cline{1-1} \cline{3-3}\cline{4-4} \cline{5-5}\cline{6-6} \cline{8-8}\cline{9-9}
\cline{10-10}\cline{11-11}
\multicolumn{11}{c}{}\\
\cline{1-1} \cline{3-3}\cline{4-4} \cline{5-5}\cline{6-6} \cline{8-8}\cline{9-9}
\cline{10-10}\cline{11-11}
\e& & \e & i & j & k & &a & a_i & a_j & a_k \\
\cline{1-1} \cline{3-3}\cline{4-4} \cline{5-5}\cline{6-6} \cline{8-8}\cline{9-9}
\cline{10-10}\cline{11-11}
i& & i & -\e & k & -j & &\lambda a_i & 0 & \half a_k &-\half a_j \\
\cline{1-1} \cline{3-3}\cline{4-4} \cline{5-5}\cline{6-6} \cline{8-8}\cline{9-9}
\cline{10-10}\cline{11-11}
j& & j & -k & -\e & i & &\mu a_j & -\half a_k  & 0 &\half a_i \\
\cline{1-1} \cline{3-3}\cline{4-4} \cline{5-5}\cline{6-6} \cline{8-8}\cline{9-9}
\cline{10-10}\cline{11-11}
k& & k & j & -i & -\e& &\nu a_k & \half a_j   & -\half a_i & 0 \\
\cline{1-1} \cline{3-3}\cline{4-4} \cline{5-5}\cline{6-6} \cline{8-8}\cline{9-9}
\cline{10-10}\cline{11-11}
\multicolumn{11}{c}{}\\
\cline{1-1} \cline{3-3}\cline{4-4} \cline{5-5}\cline{6-6} \cline{8-8}\cline{9-9}
\cline{10-10}\cline{11-11}
 a &&a  & \lambda a_i & \mu a_j& \nu a_k& &0& i & j & k \\
 \cline{1-1} \cline{3-3}\cline{4-4} \cline{5-5}\cline{6-6} \cline{8-8}\cline{9-9}
\cline{10-10}\cline{11-11}
 a_i &&  a_i & 0 &-\half a_k & \half a_j && -i&0& 0 & 0 \\
 \cline{1-1} \cline{3-3}\cline{4-4} \cline{5-5}\cline{6-6} \cline{8-8}\cline{9-9}
\cline{10-10}\cline{11-11}
  a_j &&   a_j & \half a_k & 0&-\half a_i  && -j&0& 0 & 0  \\
  \cline{1-1} \cline{3-3}\cline{4-4} \cline{5-5}\cline{6-6} \cline{8-8}\cline{9-9}
\cline{10-10}\cline{11-11}
   a_k &&   a_k & -\half a_j &\half a_i &0&& -k&0& 0 & 0 \\
   \cline{1-1} \cline{3-3}\cline{4-4} \cline{5-5}\cline{6-6} \cline{8-8}\cline{9-9}
\cline{10-10}\cline{11-11}
\end{array}
\end{displaymath}

\bigskip
\bigskip


\noindent
where $\l,\mu,\nu\in\C$ are parameters.
The obtained algebras are isomorphic if and only if the corresponding
parameters are proportional.
One thus have a two-parameter family of algebras parametrized by
$\CP^2/\tau$, where $\tau$ is the action of the cyclic group $\Z_3$ of
coordinate permutation.

It is easy to check that the algebra of derivations $\mathrm{Der}(\A)$ does not
depend on $\l,\mu,\nu$.
This algebra is $1|1$-dimensional, it has one even generator $T_0$
and one odd generator $T_1$ satisfying the commutation relations
$[T_0,T_1]=T_1$ and $[T_1,T_1]=0$.

It is of course very easy to define an analogous construction for an
arbitrary Clifford algebra.
We will not dwell on it here.

\bigskip




\begin{thebibliography}{99}

\bibitem{AM}
H. Albuquerque, S. Majid, 
{\it Quasialgebra structure of the octonions}, J. Algebra
{\bf 220} (1999), 188--224.

\bibitem{AM1}
H. Albuquerque, S. Majid, 
{\it Clifford algebras obtained by twisting of group algebras},
J. Pure Appl. Algebra  {\bf 171}  (2002), 133--148.

\bibitem{ABS}
M. Atiyah, R. Bott, A. Shapiro,
{\it Clifford modules},
Topology {\bf 3} suppl. 1 (1964), 3--38. 

\bibitem{BMZ}
Yu. Bahturin, S. Montgomery, M. Zaicev,
{\it Generalized Lie Solvability of Associative 
Algebras}, in: Proceedings of the International Workshop on 
Groups, Rings, Lie and Hopf Algebras, St. John's, Kluwer, 
Dordrecht, 2003, pp. 1--23.

\bibitem{Bah}
Yu. Bahturin, S. Sehgal, M. Zaicev,
{\it Group gradings on associative algebras}, 
J. Algebra  {\bf 241}  (2001), 677--698.

\bibitem{Bah1}
Yu. Bahturin, S. Sehgal, M. Zaicev,
{\it Finite-dimensional simple graded algebras}, 
Sb. Math., {\bf 199} (2008), 965--983. 

\bibitem{BE}
G. Benkart, A. Elduque, 
{\it A new construction of the Kac Jordan superalgebra},
Proc. Amer. Math. Soc. {\bf 130}  (2002), 3209--3217.

\bibitem{Phy}
A. de Goursac, T. Masson, J-C.Wallet
{\it Noncommutative $\varepsilon$-graded connections and application to Moyal space.}
arXiv:0811.3567

\bibitem{Eld}
A. Elduque,
{\it Gradings on octonions},
J. Algebra  {\bf 207}  (1998), 342--354.

\bibitem{Eld1}
A. Elduque,
{\it  Gradings on symmetric composition algebras},
arXiv:0809.1922.

\bibitem{McC}
K. McCrimmon,
{\it Kaplansky superalgebras}, J. Algebra {\bf 164} (1994),
656--694.

\bibitem{MG}
S. Morier-Genoud,
{\it Representations of $\asl_2$},
Intern. Math. Res. Notices., 2009.

\bibitem{MO}
S. Morier-Genoud, V. Ovsienko, 
{\it Well, Papa, can you multiply triplets?},  
Mathematical Intelligencer, to appear.

\bibitem{Ovs}
V. Ovsienko,
{\it Lie antialgebras: pr\'emices},
arXiv:0705.1629.

\bibitem{Por}
I.R. Porteous,
Clifford Algebras and the Classical Groups,
Cambridge University Press, 1995.

\end{thebibliography}
\end{document}